\documentclass{amsart}

\usepackage{dcolumn}
\usepackage{geometry}               
\usepackage{color}

\usepackage{dsfont}
\usepackage{mathtools}

\usepackage{hyperref}
\definecolor{myurlcolor}{rgb}{0,0,0.7}
\hypersetup{colorlinks,
linkcolor=myurlcolor,
citecolor=myurlcolor,
urlcolor=myurlcolor}

\newtheorem{theorem}{Theorem}

\newcommand{\define}[1]{{\bf \boldmath{#1}}}

\newcommand{\diam}{\textrm{diam}}
\newcommand{\area}{\textrm{area}}
\newcommand{\R}{\mathbb{R}}

\title{The Lebesgue Universal Covering Problem}

\date{\today}                                           

\begin{document}

\author{John C. Baez and Karine Bagdasaryan}
\address{Department of Mathematics, University of California, Riverside CA 92521, USA }
\email{baez@math.ucr.edu, karine.bagdasaryan@ucr.edu}
\author{Philip Gibbs}
\address{6 Welbeck Drive
Langdon Hills, Basildon, Essex, SS16 6BU, UK}
\email{philegibbs@gmail.com }
\maketitle

\begin{abstract} 
In 1914 Lebesgue defined a `universal covering' to be a convex subset of the plane that contains an isometric copy of any subset of diameter 1.  His challenge of finding a universal covering with the least possible area has been addressed by various mathematicians: P\'al, Sprague and Hansen have each created a smaller universal covering by removing regions from those known before.  However, Hansen's last reduction was microsopic: he claimed to remove an area of $6 \cdot 10^{-18}$, but we show that he actually removed an area of just $8 \cdot 10^{-21}$.   In the following, with the help of Greg Egan, we find a new, smaller universal covering with area less than $0.8441153$.  This reduces the area of the previous best universal covering by $2.2 \cdot 10^{-5}$.  
\end{abstract}

\section{Introduction}
 
A `universal covering' is a convex subset of the plane that covers a translated, reflected and/or rotated version of every subset of the plane with diameter 1.  In 1914, Lebesgue \cite{Hansen} laid down the challenge of finding the universal covering with least area, and since then various other mathematicians have worked on it, slightly improving the results each time.  The most notable of them are J.\ P\'al \cite{Pal}, R.\ Sprague \cite{Sprague}, and H.\ C.\ Hansen \cite{Hansen}.   More recently Brass and Sharifi \cite{Brass} have found a lower bound on the area of a universal covering, while Duff \cite{Duff} showed that dropping the requirement of convexity allows for an even smaller area.

The simplest universal covering is a regular hexagon in which one can inscribe a circle of diameter 1.  All the improved universal coverings so far have been constructed by removing pieces of this hexagon.   Here we describe a new universal covering with less area than those known so far.  

To state the problem precisely, recall that the diameter of a set of points $A$ in a metric space is defined by  
\[ \diam(A)=\sup\{d(x,y) : x,y\in A\} .\]
We define a \define{universal covering} to be a convex set of points $U$ in the Euclidean
plane $\R^2$ such that any set $A \subseteq \R^2$ of diameter $1$ is isometric to a subset of $U$.  If we write $\cong$ to mean `isometric to', we can define
the collection of universal coverings to be
\[ \mathcal{U}=\{U :  \textrm{ for all } A \textrm{ such that } \diam(A) = 1  \textrm{ there exists } B\subseteq U \textrm{ such that } B\cong A\} .\]  
The quest for smaller and smaller universal coverings is in part the search for this number:
 \[ a=\inf\{\area(U) : U \in \mathcal{U}\} .\]  

In 1920, P\'al \cite{Pal} made significant progress on Lebesgue's problem.
The regular hexagon in which a circle of diameter $1$ can be inscribed 
has sides of length $1/\sqrt{3}$.  P\'al showed that this hexagon gives a universal
covering, and thus 
\[ a \leq\frac{\sqrt3}{2}=0.86602540\dots . \]  
P\'al then took this regular hexagon and fit the largest possible regular dodecagon inside it, as shown in Figure 1.  He then proved that two of the resulting corners could be removed to give a smaller universal covering.  This implies that
\[ a \leq 2-\frac{2}{\sqrt3}=0.84529946\dots . \]
(Here and in all that follows, when we speak of `removing' a closed set $Y \subseteq \R^2$ from another closed set $X \subseteq \R^2$, we mean to take the set-theoretic difference $X - Y$ and then form its closure.  The universal covers
we discuss are all closed.)

In 1936, Sprague went on to prove that more area could be removed from another corner of the original hexagon.  This proved
\[            a \leq 0.844137708436.\] 
In 1992, Hansen \cite{Hansen} took these reductions even further by removing more 
area from two different corners of P\'al's universal covering.  He claimed that the areas that could be removed were $4\cdot 10^{-11}$  and $6 \cdot 10^{-18}$, but we have found that his second calculation was mistaken.  The actual areas removed were $3.7507 \cdot 10^{-11}$ and $8.4541 \cdot 10^{-21}$, showing
\[      a \leq 0.844137708398. \]
We have created a Java applet illustrating Hansen's universal covering \cite{Gibbs}.  This allows the user to zoom in and see the tiny, extremely narrow regions removed by Hansen.

Our new improved universal covering gives
\[        a \leq 0.8441153. \]
This is about $2.2 \cdot 10^{-5}$ less than the previous best upper bound, due to Hansen. On the other hand, Brass and Sharifi \cite{Brass} have given the best known lower bound on the area of a universal covering.  They did so by using various shapes such as a circle, a pentagon, and a triangle of diameter $1$ in a certain alignment.  They concluded that 
\[    a \geq 0.832. \]
Thus, at present we know that
\[   0.832 \leq a \leq 0.8441153. \]

In their book \textsl{Old and New Unsolved Problems in Plane Geometry and 
Number Theory}, Klee and Wagon \cite{Klee} wrote:
\vskip 1em

\begin{center}
\textbf{\dots
it does seem safe to guess that progress on [this problem], which has been painfully slow in the past, may be even more painfully slow in the future.}
\end{center}

\vskip 1em

\noindent However, our reduction of Hansen's universal covering is about a million times greater than Hansen's reduction of Sprague's.  The use of computers played an important role.  Given the gap between the area of our universal covering and Brass and Sharifi's lower bound, we can hope for even faster progress in the decades to come!

\vfill
\eject

\section{Hansen's calculation}

Let us recall Hansen's universal covering \cite{Hansen} and correct
his calculation of its area.

\begin{center}
\href{http://citeseerx.ist.psu.edu/viewdoc/download?doi=10.1.1.167.6499\&rep=rep1\&type=pdf}
{
\includegraphics[scale=.15]{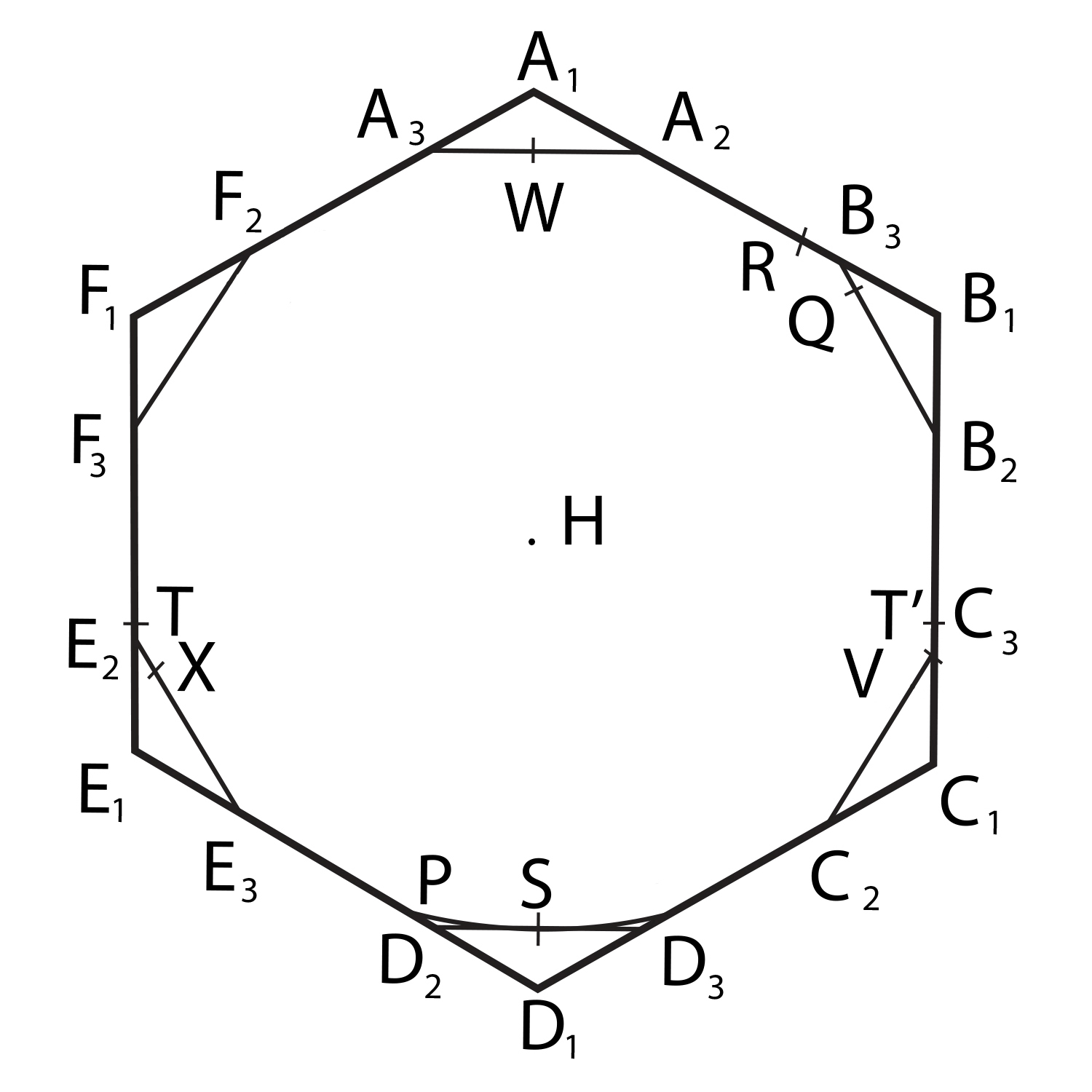}
}

\label{fig1}
1.  Figure similar to the one used by Hansen
\end{center}

\vskip 1em 
Figure 1 shows a regular hexagon $A_1B_1C_1D_1E_1F_1$ in which can
be inscribed a circle of diameter 1.  Inside this is a regular dodecagon, 
$A_3A_2B_3B_2C_3C_2D_3D_2E_3E_2F_3F_2$.  P\'al showed that one
can take the hexagon, remove the triangles $C_1C_2C_3$ and
$E_1E_2E_3$, and still obtain a universal covering.   We generalize his proof in Reduction 1 of Section \ref{improved}.

Later, Sprague \cite{Sprague} constructed a smaller universal covering.  To do this he showed that near $A_1$ the region outside the circle of radius $1$ centered at $E_3$ could be removed, as well as the region outside the circle of radius $1$ centered at $C_2$.  We generalize this argument in Reduction 2.

Building on Sprague's work, Hansen constructed a still smaller universal covering.  First he removed a tiny region $XE_2 T$, almost invisible in this diagram.  Then he removed a much smaller region $T'C_3V$.  Each of these regions is a thin sliver bounded by two line segments and an arc, as shown in our Java applet \cite{Gibbs}.

\vfill \break

To compute the areas of these regions, Hansen needed to determine certain distances $x_0, \dots , x_4$.  The points and arcs he used, and these distances, are defined as follows:

\begin{itemize}
\item The distance from $W$, the midpoint of $A_2A_3$, to the dodecagon corner $A_3$ is $x_0$.
\item The circle of radius $1$ centered at $W$ is tangent to $D_2D_3$ at $S$ and hits the dodecagon at $P$.  The distance from $S$ to the dodecagon corner $D_2$ can be seen to equal $x_0$.   The distance from $D_2$ to $P$ is $x_1$.
\item The circle of radius $1$ centered at $P$ is tangent to $B_3A_2$ at $R$ and hits the dodecagon at $Q$.  The distance from $R$ to the dodecagon corner $B_3$ can be seen to equal $x_1$.  The distance from $B_3$ to $Q$ is $x_2$.  
\item The circle of radius $1$ centered at $Q$ is tangent to $E_2E_3$ at $X$ and hits the dodecagon at $T$. The distance from $X$ to the dodecagon corner $E_2$ can be seen to equal $x_2$.  The distance from $E_2$ to $T$ is $x_3$.
\item The circle of radius $1$ centered at $T$ is tangent to $C_3B_2$ at $T'$ and hits the dodecagon at $V$.  The distance from $T'$ to the dodecagon corner $C_3$ can be seen to equal $x_3$.  The distance from $C_3$ to $V$ is $x_4$.
\end{itemize}

Here are the two regions Hansen removed:

\begin{itemize}
\item
Hansen's first region $XE_2T$ is bounded by the line segment $XE_2$, the much shorter line segment $E_2T$, and the arc $XT$ that is part of the circle of radius $1$ centered at $Q$.  He claimed this first region has an area $4 \cdot 10^{-11}$.  This calculation is approximately right. 
\item
Hansen's second region $T'C_3V$ is bounded by the line segment $T'C_3$, the much shorter line segment $C_3V$, and the arc $T'V$ that is part of the circle of radius $1$ centered at $T$.  He claimed this second region had an area of $6 \cdot 10^{-18}$.   This calculation is far from correct.
\end{itemize}

\noindent
In order to find the correct areas, we first calculate $x_0$ and then recursively calculate the other distances $x_i$.  

Let $S$ be the midpoint of $D_2 D_3$.  Recall that $x_0$ is the distance from $S$ to $D_2$, while $x_1$ is the distance from $D_2$ to $P$.   We compute these distances using some elementary geometry. To do this we must keep in mind some basic facts.  The length of each side of the hexagon is  $\frac{1}{\sqrt{3}}$.  The circle inscribed in the hexagon has radius $\frac{1}{2}$.  The arc from $S$ to $P$ has a radius of $1$ and is tangent to $SD_2$ at $S$.  

To compute the distance $x_0$ we first need to determine the distance $SD_1$.  
Let $H$ be the center of the circle, as in Figure 1.  The distance $SD_1$ is the distance $HD_1$ minus the distance $HS$, that is, $\frac{1}{\sqrt{3}}-\frac{1}{2}$.  
Since $D_2SD_1$ is a 30--60--90 triangle, the length $x_0 = D_2S$ is found $\sqrt{3}$ times $SD_1$.  Thus, 
\[   x_0 = 1 - \frac{\sqrt{3}}{2} .\]

Knowing $x_0$ we can then calculate $x_1$.   More generally, knowing $x_i$ we can calculate $x_{i+1}$. The diagram in Figure 2 can be used to solve for $x_{i+1}$, and further to calculate the area of the region $IJK$ bounded by the line $IJ$, the line $IK$, and the arc through $I$ and $K$ that is part of a circle of radius $1$ centered at $A$.  The two regions Hansen removed, namely $XE_2T$ and $T'C_3V$, are examples of this type of region $IJK$.

\vfill \eject
\begin{center}
{
\includegraphics[scale=.12]{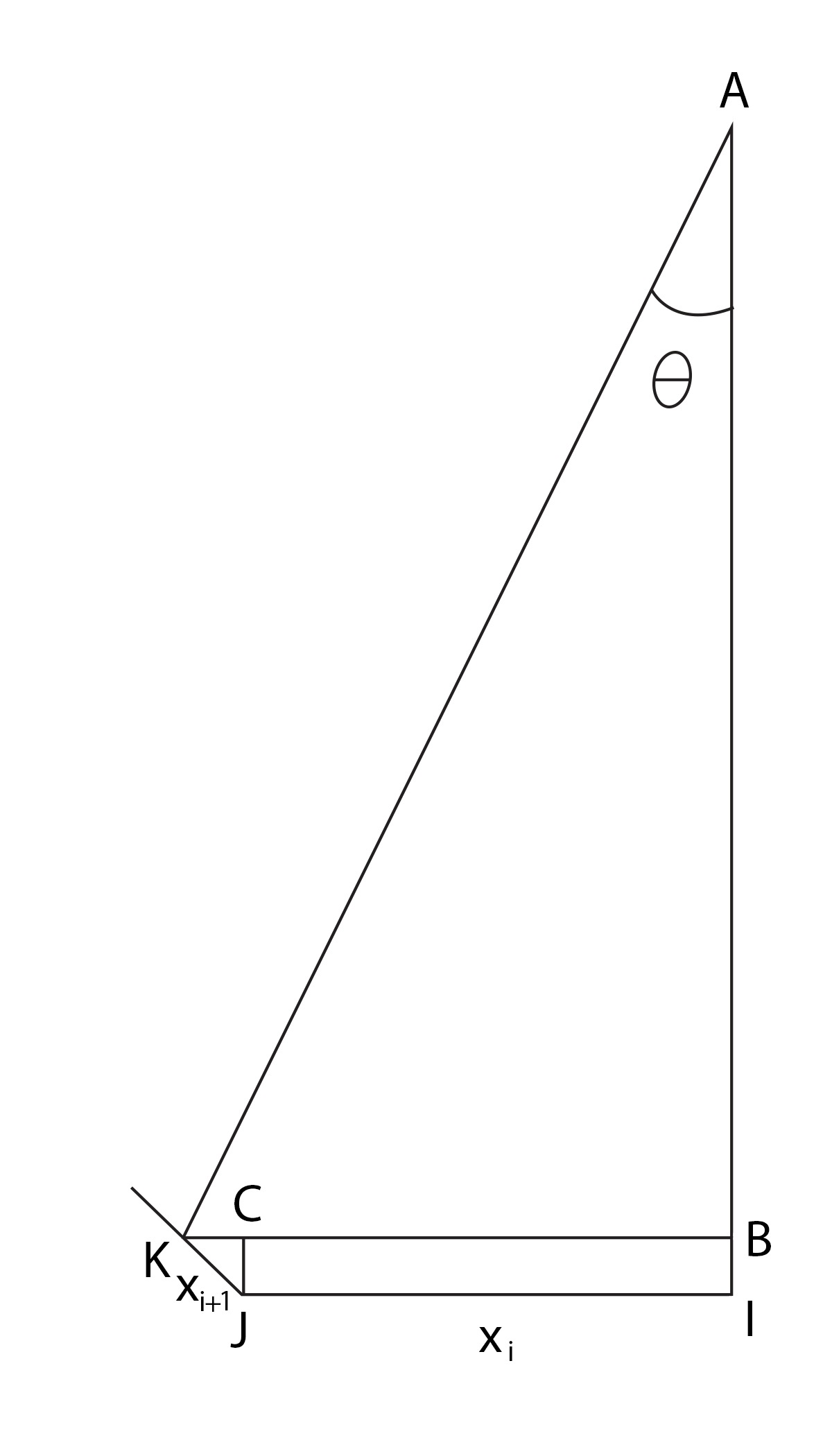}
}

\label{fig2}
2.  Diagram used to solve for $x_i$ and $x_{i+1}$
\end{center}
\vskip 1em 

In Figure 2 we let $AI$ and $AK$ be line segments of length $1$.  We let $AI$ be perpendicular to $JI$ and $KB$ and parallel to $CJ$. Let $x_i$ be the distance from $J$ to $I$ and $x_{i+1}$ be the distance from $K$ to $J$; $BI=CJ$.  $KI$ is an arc of a circle centered at $A$ with radius $1$.  We set $\angle IJK ={150}^\circ$, the interior angle of a regular dodecagon.

Note that $\angle IJC ={90}^\circ$ and $\angle CJK ={60}^\circ$.  The triangle $CJK$ is a right triangle with  $\angle CKJ ={30}^\circ$.  We thus obtain
\[BI=CJ=\frac{x_{i+1}}{2}, \quad CK=\frac{\sqrt{3}x_{i+1}}{2}, \quad BC =x_i .\]
Then by applying the Pythagorean theorem to the  triangle $ABK$, we see
\[ 1= (AK)^2=(BK)^2+(AB)^2=\left(CK+BC \right)^2+\left(1-BI\right)^2=\left(\frac{\sqrt{3}x_{i+1}}{2}+x_i \right)^{\!\! 2}+\left(1-\frac{x_{i+1}}{2}\right)^2. \]  
This gives a quadratic relationship between $x_i$ and $x_{i+1}$ which can be solved for $x_{i+1}$:
\[  x_{i+1} =\frac{1 - \sqrt{3} x_i - \sqrt{1 - 2 \sqrt{3}\,x_i - x_i^2}}{2} .\]
To avoid rounding errors we multiply the numerator and denominator by $1 - \sqrt{3} x_i + \sqrt{1 - 2 \sqrt{3}x_i - x_i^2}$.  This leaves us with
\[   x_{i+1} = \frac{2 x_i^2}{1 - \sqrt{3} x_i + \sqrt{1 - 2 \sqrt{3}x_i - x_i^2}} .\]

Let $a_i$ be the area of the region $IJK$ bounded by the line segments $IJ$, $JK$ and the arc centered at $A$ going through $I$ and $K$.  To compute $a_i$, we first calculate the area of the triangle $IJK$ and then subtract the area of the region between the line segment $IK$ and the arc through $I$ and $K$.  The area of the triangle $IJK$ is 
\[        \frac{IJ\cdot CJ}{2}=\frac{x_i x_{i+1}}{4} . \]
The area of the other region is 
\[            \frac{R^2}{2}(\theta -\sin\theta), \] 
where $\theta$ is the angle $\angle BAK$ and $R$ is the radius of the circle, namely $AK$, which is $1$.  We determine $\theta$ using the formula $\theta=2\sin^{-1}(d/2R)$, where $d$ is the length of the chord $IK$.    We thus obtain
\[  a_i = \frac{x_ix_{i+1}}{4} - \frac{(\theta - \sin\theta)}{2} \]
where $\theta = 2\sin^{-1}(\frac{d}{2})$ and  
\[d = \sqrt{\frac{{x_{i+1}}^2}{4} + \left(x_i + \frac{\sqrt{3}}{2}x_{i+1}\right)^{\!\!2}} \; .\]

Table 1 shows the lengths $x_i$ and areas $a_i$.  We see that the area of the first region Hansen removed, namely $XE_2 T$, is $a_2 \approx 3.7507 \cdot 10^{-11}$.
This is close to Hansen's claim of $a_2 = 4 \cdot 10^{-11}$.   However, we see that the area of the second region Hansen removed, namely $T'C_3V$, is $a_3 \approx 8.4541 \cdot 10^{-21}$.   This is significantly smaller than Hansen's claim of $a_3 = 6 \cdot 10^{-18}$.

\vskip 2em

{\vbox{   
\begin{center}   
{\small
\setlength{\extrarowheight}{7pt}
\setlength{\tabcolsep}{5pt}
\begin{tabular}{|c|l|l|}    \hline
$i$ & $x_i$  & $a_i$                                    \\ \hline
0& $1.339745962156 \cdot 10^{-1}$  & $4.952913815765 \cdot 10^{-4}$   \\
1& $2.413116066646 \cdot 10^{-2}$  & $2.418850424555 \cdot 10^{-6}$   \\
2& $6.080990483915 \cdot 10^{-4}$  & $3.750723412843 \cdot 10^{-11}$ \\
3& $3.701744790810 \cdot 10^{-7}$  & $8.454119457933 \cdot 10^{-21}$ \\
4& $1.370292328207 \cdot 10^{-13}$& $4.288332272809 \cdot 10^{-40}$ \\
\hline
\end{tabular}} 
\end{center}   
\vskip 1em 
\centerline{Table 1: Lengths $x_i$ and areas $a_i$ in Hansen's construction}
}}   

 \eject

\section{An improved universal covering}
\label{improved}

\begin{theorem}
 There exists a universal covering with area less than or equal to $0.8441153$.
\end{theorem}

\begin{proof}

We begin with the regular hexagon in Figure 1.   We shall modify this universal covering in the following ways:

\vskip 1em 
\subsubsection*{Reduction 1.} 

\begin{center}
{
\includegraphics[scale=.6]{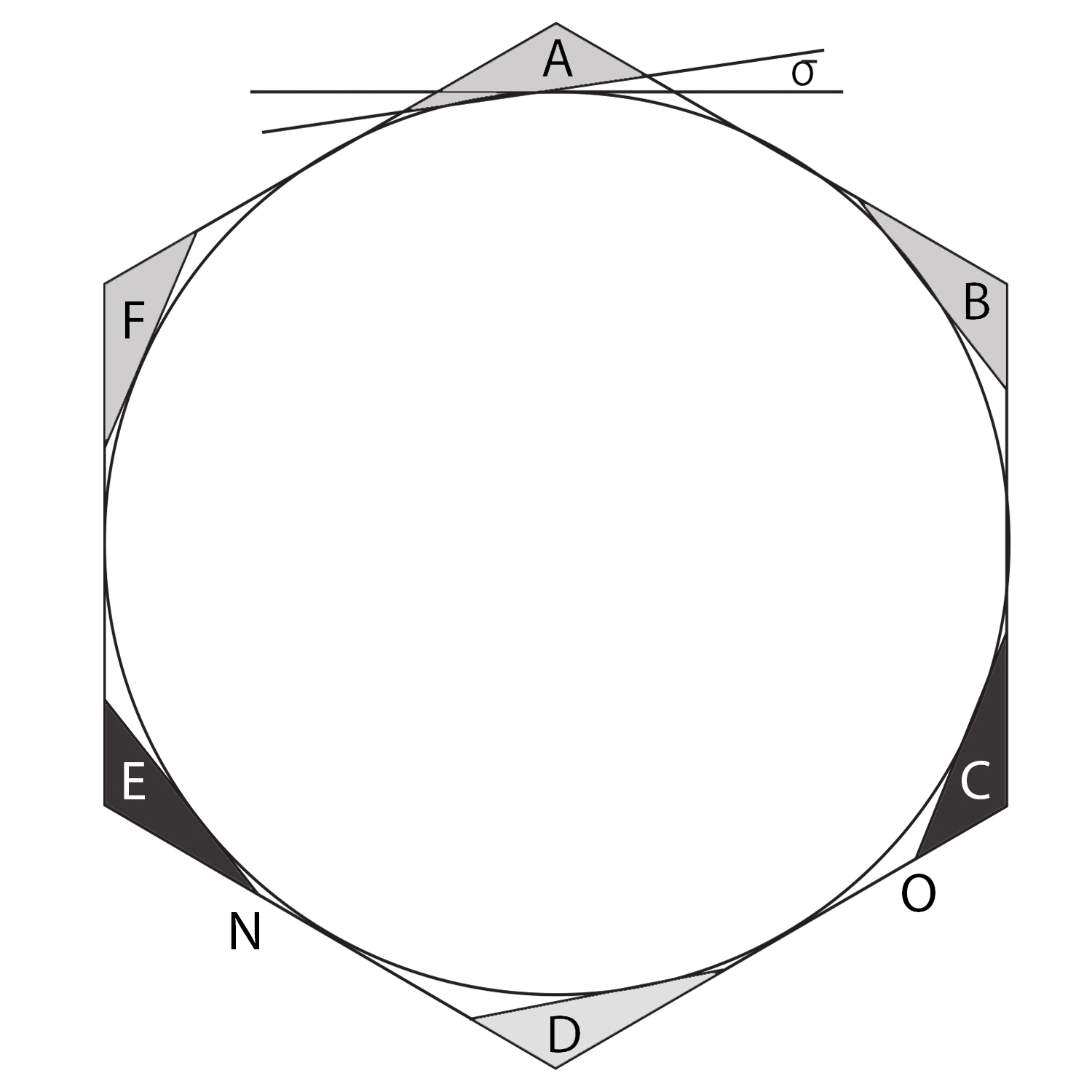}

}

\label{fig3}
3.  Reduction with slant angle $\sigma$ 
\end{center}
\vskip 1em 

Start with a regular hexagon $\mathcal{H}$ with a circle of diameter $1$ inscribed inside of it.  P\'al proved that this is a universal covering \cite{Pal}.  Let $\mathcal{H'}$ be a 
regular hexagon of the same size centered at the same point, rotated counterclockwise
by an angle of $30 + \sigma$ degrees.  

When $\sigma = 0$, the case considered by P\'al, the intersection $\mathcal{H} \cap \mathcal{H'}$ is the largest possible regular dodecagon contained in the original 
hexagon $\mathcal{H}$.  As shown in Figure 1, this dodecagon has corners $A_3A_2B_3B_2C_3C_2D_3D_2E_3E_2F_3F_2$.  P\'al proved
that if we remove the triangles $C_1C_2C_3$ and $E_1E_2E_3$ from the hexagon
$\mathcal{H}$, the remaining set is still a universal covering.

We generalize this to the case where $\sigma$  is some other small angle---say, less
than $10^\circ$.  Later we shall determine that the optimal value of $\sigma$ for our purposes is about $1.29^\circ$.  Note that $\mathcal{H}$ with $\mathcal{H}'$ removed is the union of six triangles $A,B,C,D,E,F$ shown in Figure 3.  Let $\mathcal{K}$ be $\mathcal{H}$ with $C$ and $E$ removed.  We claim that $\mathcal{K}$ is a universal covering.

To prove this, we generalize an argument due to to P\'al \cite{Pal}.  Imagine two parallel lines each tangent to the circle, containing two opposite edges of the rotated hexagon $\mathcal{H}'$.  Any pair of points inside $\mathcal{H}$ and on opposite sides of $\mathcal{H}'$ have a distance of at least $1$ from each other, since the diameter of the circle is $1$.  Consider the six triangles $A,B,C,D,E,F$.  An isometric copy of any set of diameter $1$ can be fit inside $\mathcal{H}$.  Since the distance between any two triangles that are opposite each other is $1$, this set cannot simultaneously have points in the interior of two opposite triangles. Thus, this set can only have a nonempty intersection with three adjacent triangles or three nonadjacent triangles. In the first case we can rotate the set so that the only triangles it intersects are $F$, $A$ and $B$.  In the second case we can rotate it so that the only triangles it intersects are $F$, $B$ and $D$.  In either case, it fails to intersect the triangles $E$ and $C$.  Thus, the set $\mathcal{K}$ consisting of the hexagon $\mathcal{H}$ with these two triangles removed is a universal covering. 

\vskip 1em 
\subsubsection*{Reduction 2.}

Recall that a \define{curve of constant width} is a convex subset of the plane whose \define{width}, defined as the perpendicular distance between two distinct parallel lines each having at least one point in common with the sets boundary but none with the set's interior, is independent of the direction of these lines.  Vre\'cica \cite{Vrecica} has shown that any subset of the plane with diameter $1$ can be extended to a curve of constant width $1$.   Thus, a set will be a universal covering if it contains an isometric copy of every curve of constant width $1$.    

Consider any curve of constant width $1$ inside $\mathcal{K}$.  It must touch each of the hexagon's sides at a unique point.  To see this, note that if it does not touch one of the sides, it would have width less than $1$.   On the other hand, if it touched a side at two points, then it would have width greater than $1$. 

The curve must touch the side of $\mathcal{H}$ running from $D$ to $C$ at some point to the left of $O$, the corner of the removed triangle $C$.  This implies that all points near $A$ outside an arc of radius $1$ centered on $O$ can be removed from the closure of $\mathcal{K}$, obtaining a smaller universal covering.

Similarly, the curve must touch the side of the hexagon running from $E$ to $D$ somewhere to the right of $N$, the corner of the removed triangle $E$.  Thus, all points near $A$ outside an arc of radius $1$ centered on $N$ can also be removed.  The remaining universal covering then has a vertex at the point $X$ where these two arcs meet; for a closeup see Figure 4. 

This reduces the area of the univeral covering, but not enough to make it smaller for any nonzero value of $\sigma$ than Sprague's universal covering, which is the case $\sigma=0$.   So, we must go further.

\vskip 2em

\begin{center}
{
\includegraphics[scale=.6]{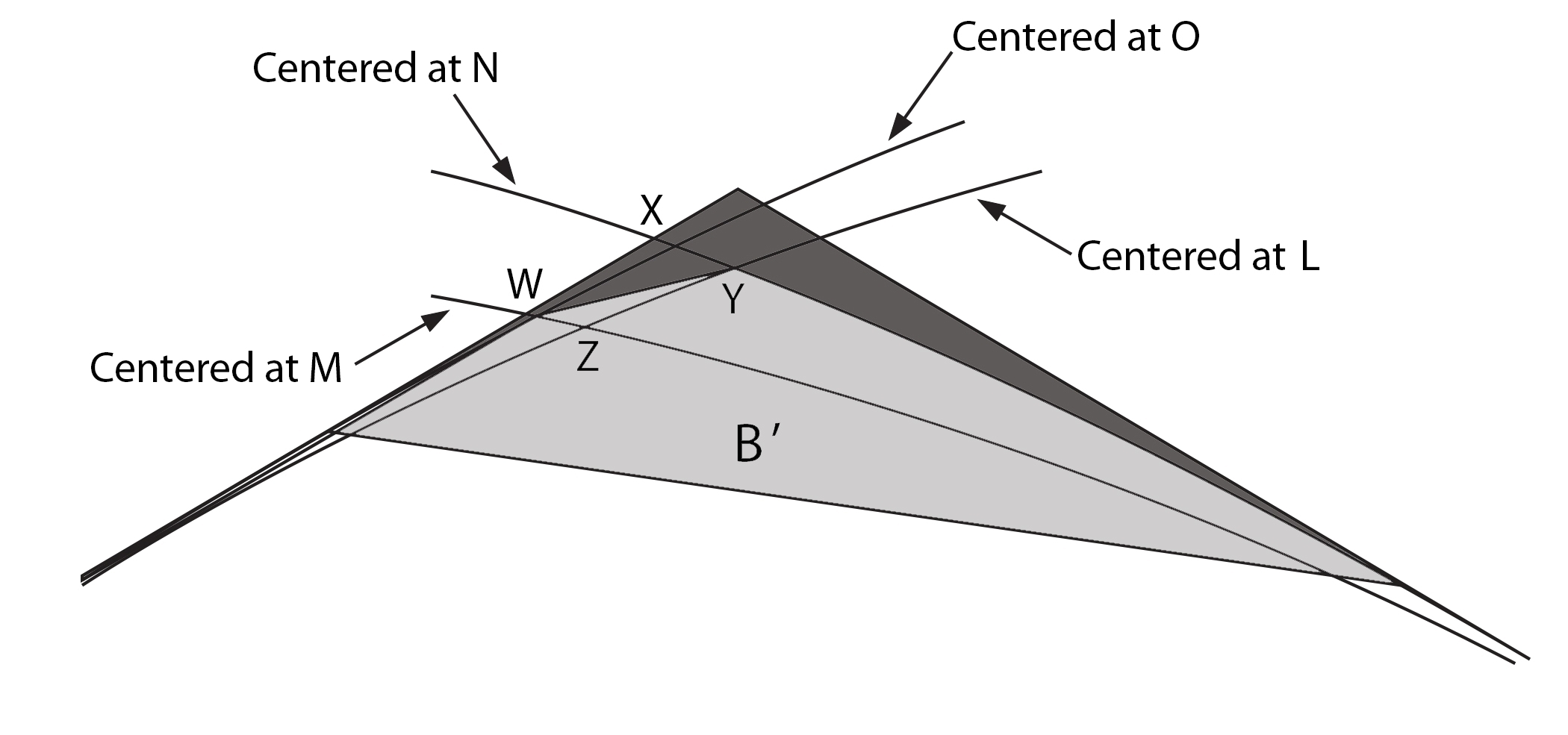}
}

\label{fig4}
4.  Reduction closeup
\end{center}

\subsubsection*{Reduction 3.} 
For the final reduction we refer to Figures 4 and 5.   Figure 4 shows a region $WXY$ bounded by two arcs $WX$ and $XY$ and a straight line segment $WY$ defined as follows:

\begin{itemize}
\item The arc through $W$ and $X$ is the arc of radius $1$ centered at $O$, which is a point where the rotated hexagon $\mathcal{H'}$ intersects the original hexagon $\mathcal{H}$.

\item $X$ is the point of intersection of the arc of radius $1$ centered at $O$ and the arc of radius $1$ centered on $N$, which is another point where the rotated hexagon intersects the original hexagon. 

\item $W$ is the intersection of the arc centered at $O$ with the arc of radius $1$ centered on $M$, the midpoint of the edge $E_1D_1$ of the original hexagon.

\item $Y$ is the intersection of the arc with radius $1$ centered at $N$ and the arc with radius $1$ centered at $L$, which is another point where the rotated hexagon intersects the original hexagon.
\end{itemize}

We claim that the univeral covering considered in Reduction 2 can be further reduced by removing the region $WXY$ for a specific angle $\sigma$ that we will specify. 
(In fact, we could remove the whole of region $WXYZ$ outside the arcs $WZ$ and $ZY$ and still be left with a set that contains an isometric copy of every curve of constant width 1.  However, this set would not be convex, since it contains $W$ and $Y$ but not points on the line segment between $W$ and $Y$.  Since a universal covering is required to be convex, we only remove the smaller region $WXY$, one of whose sides is the straight segment $WY$.)

To prove our claim, consider Figure 5.   This shows an axis of symmetry of the hexagon $\mathcal{H}$ going through a point $M$, the midpoint of the side $D_1E_1$.  When the triangles $A, B, C, D, E, F$ are reflected about this axis they are mapped to new triangles $A', B', C', D', E', F'$, some of which are shown in gray in this figure.

\begin{center}
{
\includegraphics[scale=.6]{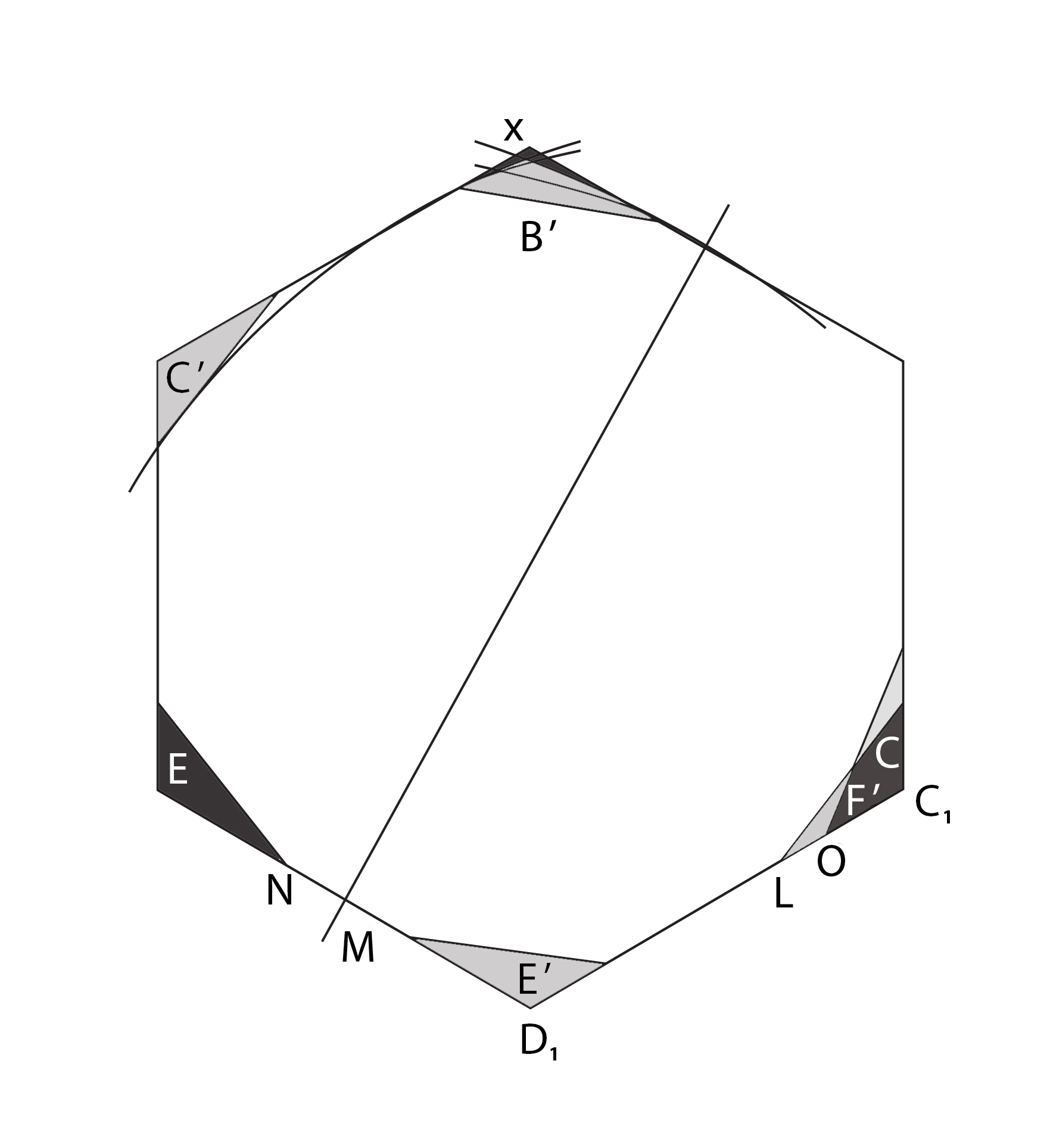}
}

\label{fig5}
5.  Reduction with symmetry
\end{center}

To prove our claim, it suffices to show that any curve of constant width $1$ can be positioned inside the hexagon with the corners removed at $E$ and $C$ in such a way that no point enters the region $WXY$.   To show this, we separately consider three cases.  Case (1) is where the curve of constant width $1$ enters the interior of triangle $E'$. Case (2) is where the curve does not enter the interior of $E'$ but it does enter the interior of region $C'$. Case (3) is where the curve does not enter the interior of $E'$ or $C'$.  

Case (1): If a curve of constant width $1$ enters the interior of $E'$ then it cannot enter the interior of $B'$ since all points inside $B'$ are at a distance greater than $1$ from points inside $E'$. It is therefore sufficient to show that all points in the region $WXY$ are inside $B'$.  This can be checked by calculation for the specific slant angle $\sigma$ to be used.

Case (2): If a curve $\mathcal{Q}$ of constant width $1$ enters the interior of $C'$ then it cannot enter the interior of the region $F'$. It must therefore touch the side $D_1C_1$ between $D_1$ and $L$. It follows that no point inside $\mathcal{Q}$ can lie outside the arc of radius $1$ centered at $L$ within the triangle $A$.  

Therefore, no point of $\mathcal{Q}$ will lie inside the region $WXY$ provided the interior of this region lies outside the arc of radius $1$ centered at $L$.  This condition holds if no points on the segment $WY$ lie inside the arc of radius $1$ centered at $L$.  This in turn follows from $\angle{WYL}$ being greater than a right angle.  We shall check this by a calculation for the specific angle $\sigma$ we use.

Case (3): If a curve $\mathcal{Q}$ of constant width $1$ does not enter the interior of $E'$ or $C'$ then it can be reflected about the axis through $M$ and the center of the hexagon, giving a curve $\mathcal{Q}'$ also of constant width 1.  This reflection maps $C'$ to $C$ and $E'$ to $E$, so $\mathcal{Q}'$ will lie inside the hexagon with the two corners $C$ and $E$ removed. Recall from Reduction 1 that $\mathcal{Q}$ has no points in $C$ or $E$.  Thus $\mathcal{Q}'$ has no points in $C'$ or $E'$. 

Each of the curves $\mathcal{Q}$ and $\mathcal{Q}'$ touch the side of the hexagon between $N$ and $N'$ at a single point. If $\mathcal{Q}$ touches this side between $N$ and $M$ then $\mathcal{Q}'$ will touch this side between $M$ and $N'$, and vice versa. Since reflections are an allowed isometry, we are free to use the curve in either position $\mathcal{Q}$ or $\mathcal{Q}'$, so we choose to use whichever of the two touches the side between $M$ and $N'$. No point inside the curve will then lie both outside the arc of radius $1$ centered at $M$ and to the left of the line of reflection.

Therefore, no point of the chosen curve will lie inside the region $WXY$ provided that no point in the interior of this region lies inside the arc of radius $1$ centered at $M$. This condition holds if no points on the segment $WY$ lie inside the arc of radius $1$ centered at $M$.  This is turn follows from $\angle{MWY}$ being greater than a right angle.  We shall check this by a calculation for the specific angle $\sigma$ we use.

Our computation of the area of the resulting universal covering using Java is available online \cite{Java program}, along with the output \cite{Java output}.  The output lists
various choices of $\sigma$, followed by the corresponding values of the area of the
would-be universal covering, together with the angles $\angle{WYL}$ and  $\angle{MWY}$.  We find that when $\sigma = 1.3^\circ$, the area of the universal covering is 
\[    0.8441153768593765 \]
within the accuracy provided by double-precision floating-point arithmetic. 
For this value of $\sigma$, $\angle{WYL}$ is approximately $90.00593^\circ$
and $\angle{MWY}$ is approximately $122.9277^\circ$.  We could obtain a smaller
area if $\sigma$ were smaller, but if were too small the constraint that $\angle{WYL}$ exceed a right angle would be violated.

Greg Egan has checked our calculations using Mathematica with high-precision arithmetic (2000 digits working precision).  He found that an angle 
\[  \sigma = 1.294389444703601012^\circ \]
still obeys the necessary constraints, giving a universal covering with area
\[   0.844115297128419059\dots.\]
Egan has kindly made his Mathematica notebook available online \cite{Egan notebook}, along with a printout that is readable without this software \cite{Egan printout}.

With Egan's help, we thus claim to have proved that
\[     a \le 0.8441153.   \qedhere \]
\end{proof}

\vskip 5em
\subsection*{Acknowledgements}

We thank Greg Egan for catching some mistakes in this paper, and especially for verifying and improving our result.  We also thank both referees for spotting and helping us fix various errors.


\vskip 2em
\begin{thebibliography}{9}

\bibitem{Java program}
J.\ Baez, K.\ Bagdasaryan and P.\ Gibbs, Java program:
\href{http://math.ucr.edu/home/baez/mathematical/lebesgue.java}{http://math.ucr.edu/home/baez/mathematical/lebesgue.java}.

\bibitem{Java output}
J.\ Baez, K.\ Bagdasaryan and P.\ Gibbs, Output for Java program: \hfill \break
\href{http://math.ucr.edu/home/baez/mathematical/lebesgue\_java\_output.txt}{http://math.ucr.edu/home/baez/mathematical/lebesgue\_java\_output.txt}.

\bibitem{Brass}
P.\ Brass and M.\ Sharifi, \href{http://www.cs.cmu.edu/~mehrbod/UC05.pdf}{A lower bound for Lebesgue's universal cover problem}, \textsl{International Journal of Computational Geometry \& Applications} \textbf{15} (2005), 537--544.

\bibitem{Duff}
G.\ F.\ D.\ Duff, A smaller universal cover for sets of unit diameter, \textsl{C.\ R.\ Math.\ Acad.\ Sci.\ }\textbf{2} (1980), 37--42.

\bibitem{Egan notebook}
G.\ Egan, Mathematica notebook: \href{http://math.ucr.edu/home/baez/mathematical/lebesgue\_mathematica.nb}{http://math.ucr.edu/home/baez/mathematical/lebesgue\_mathematica.nb}.

\bibitem{Egan printout}
G.\ Egan, Mathematica printout:
\href{http://math.ucr.edu/home/baez/mathematical/lebesgue\_mathematica.pdf}{http://math.ucr.edu/home/baez/mathematical/lebesgue\_mathematica.pdf}.

\bibitem{Gibbs}
P.\ Gibbs, Lebesgue's universal
covering problem, Java applet at \href{http://gcsealgebra.uk/lebesgue/hansen/}{http://gcsealgebra.uk/lebesgue/hansen/}.

\bibitem{Hansen}
H.\ Hansen, \href{http://citeseerx.ist.psu.edu/viewdoc/download?doi=10.1.1.167.6499\&rep=rep1\&type=pdf}{Small universal covers for sets of unit diameter}, \textsl{Geom.\ Ded.\ } \textbf{42} (1992), 205--213.

\bibitem{Klee}
V. Klee and S.\ Wagon, \textsl{Old and New Unsolved Problems in Plane Geometry and Number Theory}, Cambridge U.\ Press, 1991.

\bibitem{Pal}
J.\ P\'al, \href{http://www.sdu.dk/media/bibpdf/Bind\%201-9\%5CBind\%5Cmfm-3-2.pdf}{Ueber ein elementares Variationsproblem}, \textsl{Danske Mat.-Fys.\ Meddelelser III} \textbf{2} (1920), 1--35.

\bibitem{Vrecica}
S.\ Vre\'cica, \href{http://elib.mi.sanu.ac.rs/files/journals/publ/49/n043p289.pdf}{A note on curves of constant width}, \textsl{Publications de L'Institut Math\'ematique} {\bf 29} (1981), 289--291.

\bibitem{Sprague}
R.\ Sprague, \href{http://math.ucr.edu/home/baez/mathematical/sprague/sprague.pdf}{\"Uber ein elementares Variationsproblem}, \textsl{Matematiska Tidsskrift Ser.\ B} (1936), 96--99.
\end{thebibliography}
\end{document}